\documentclass[11pt]{amsart}

\usepackage[a4paper]{geometry}
\usepackage{amsfonts,amssymb,amscd,amstext}
\usepackage{graphicx}
\usepackage[dvips]{epsfig}
\usepackage{stmaryrd}
\usepackage{enumerate}
\usepackage{titlesec}

\newcommand{\esca}[1]{\langle {#1} \rangle}
\newcommand{\doble}[1]{\ll \hspace{-0.1cm} {#1} \hspace{-0.1cm} \gg}
\newcommand{\curvo}[1]{\prec \hspace{-0.1cm} {#1} \hspace{-0.1cm} \succ}
\newcommand{\mini}[1]{\llbracket {#1} \rrbracket}

\def\Ccal{\mathcal{C}}

\def\Ncal{\mathcal{N}}

\def\Hcal{\mathcal{H}}

\def\Acal{\mathcal{A}}
\def\Fcal{\mathcal{F}}

\def\c{\mathbb{C}}

\def\R{\mathbb{R}}\def\r{\mathbb{R}}
\def\Z{\mathbb{Z}}\def\z{\mathbb{Z}}
\def\n{\mathbb{N}}
\def\s{\mathbb{S}}

\def\k{\mathbb{K}}
\def\hgot{\mathfrak{h}}
\def\Zgot{\mathfrak{Z}}
\def\Ygot{\mathfrak{Y}}
\def\ngot{\mathfrak{n}}

\def\vgot{\mathfrak{v}}

\titleformat{\section}
{\bfseries\large} {\thesection{.}}{0.2cm}{}
\titleformat{\subsection}
{\bfseries} {\thesubsection{.}}{0.15cm}{}
\titleformat{\subsubsection}
{\em}{\thesubsubsection{.}}{0.15cm}{}

\newtheorem{theorem}{Theorem}[section]
\newtheorem{proposition}[theorem]{Proposition}
\newtheorem{claim}[theorem]{Claim}
\newtheorem{lemma}[theorem]{Lemma}
\newtheorem{corollary}[theorem]{Corollary}
\newtheorem{remark}[theorem]{Remark}
\newtheorem{definition}[theorem]{Definition}

\theoremstyle{definition}

\numberwithin{equation}{section}
\numberwithin{figure}{section}

\begin{document}

\thispagestyle{empty}

\vspace*{1cm}
\noindent{\bf\LARGE Properness of associated minimal surfaces}

\vspace*{0.5cm}
\noindent{\large\bf Antonio Alarc\'{o}n$\;$ and$\;$ Francisco J. L\'{o}pez}

\footnote[0]{\vspace*{-0.4cm}

\noindent A. Alarc\'{o}n

\noindent Departamento de Geometr\'{\i}a y Topolog\'{\i}a, Universidad de Granada, E-18071 Granada, Spain

\noindent e-mail: {\tt alarcon@ugr.es}

\vspace*{0.1cm}

\noindent F.J. L\'{o}pez

\noindent Departamento de Geometr\'{\i}a y Topolog\'{\i}a, Universidad de Granada, E-18071 Granada, Spain

\noindent e-mail: {\tt fjlopez@ugr.es}

\vspace*{0.1cm}

\noindent Both authors are partially supported by MCYT-FEDER research projects MTM2007-61775 and MTM2011-22547, and Junta de Andaluc\'{i}a Grant P09-FQM-5088}

\vspace*{1cm}

{\small
\noindent {\bf Abstract}\hspace*{0.1cm} 
For any open Riemann surface $\Ncal$ and finite subset
$\Zgot\subset \s^1=\{z\in\c\,|\;|z|=1\},$ there exist an infinite closed set $\Zgot_\Ncal \subset \s^1$ containing $\Zgot$ and a null holomorphic curve
$F=(F_j)_{j=1,2,3}:\Ncal\to\c^3$ such that the map 
\[
\Ygot:\Zgot_\Ncal\times\Ncal \to \r^2,\quad\Ygot(\vgot,P)={\rm Re}\big(\vgot (F_1,F_2)(P)\big),
\]
is proper.

In particular, ${\rm Re}\big(\vgot F\big):\Ncal \to\r^3$  is a proper conformal minimal immersion properly projecting into $\r^2\equiv \r^2\times\{0\}\subset \r^3,$ for all $\vgot \in \Zgot_\Ncal.$

\vspace*{0.1cm}

\noindent{\bf Keywords}\hspace*{0.1cm} Null holomorphic curves,
associated minimal surfaces.

\vspace*{0.1cm}

\noindent{\bf Mathematics Subject Classification (2010)}\hspace*{0.1cm} 53A10; 32H02, 53C42.
}


\section{Introduction}\label{sec:intro}


Given an open Riemann surface $\Ncal,$ a conformal minimal immersion
$X:\Ncal\to\r^3$ is said to be flux-vanishing if the
conjugate immersion $X^*:\Ncal \to \r^3$ is well defined, or
equivalently, if $X$ is the real part of a {\em null holomorphic curve}
$F:\Ncal\to\c^3$ (see Definition \ref{def:null}). In this case, the family of isometric associated
minimal immersions $X_\vgot\equiv{\rm Re}(\vgot F):\Ncal\to\r^3,$ $\vgot\in\s^1=\{z \in \c\,|\; |z|=1\},$
is well defined. Notice that $X=X_1$ and recall that $X^*=X_{-\imath},$ $\imath=\sqrt{-1}.$

The aim of this paper is to study the interplay between
topological properness and associated minimal surfaces. Not so many years ago, it was a general thought that
properness strongly influences the underlying conformal structure
of minimal surfaces in $\r^3.$ In this line, Schoen and Yau asked whether there exist hyperbolic minimal surfaces in $\r^3$ properly projecting into $\r^2\equiv \r^2\times\{0\}\subset \r^3$ \cite[p. 18]{s-y}. A complete answer to this question can be found in \cite{al1}, where examples with arbitrary conformal structure
and flux map are shown. 

On the other hand, any flux-vanishing minimal surface all whose associated surfaces {\em uniformly} properly project into $\r^2$ is parabolic, see Proposition \ref{pro:sharp}. This suggests a correlation between  properness of associated surfaces and  conformal structure of minimal surfaces. The following questions arise:
{\em
\begin{enumerate}[{\sf ({Q}1)}] 
\item[{\sf ({Q}1)}] Do there exist hyperbolic flux-vanishing minimal surfaces $S$ such that both $S$ and its conjugate surface $S^*$ properly project into $\r^2$? 

\item[{\sf ({Q}2)}] More generally, how many associated surfaces of a hyperbolic flux-vanishing minimal surface can properly project into $\r^2$?
\end{enumerate}
}

Motivated by the above questions, this paper deals with those 
subsets $\Zgot \subset \s^1$ allowing proper projections in a uniform way, accordingly to the following

\begin{definition}\label{def:proj}
A closed subset $\Zgot \subset \s^1$ is said to be a projector set for an open Riemann surface $\Ncal$ if there exists a null holomorphic curve $F=(F_j)_{j=1,2,3}:\Ncal \to \c^3$ such that
the map
\[
\Ygot:\Zgot\times\Ncal \to \r^2,\quad\Ygot(\vgot,P)={\rm Re}\big(\vgot (F_1,F_2)(P)\big),
\]
is proper.

Moreover, $\Zgot$ is said to be a universal projector set if it is a projector set for any open Riemann surface.
\end{definition}

If $\Zgot$ is a projector set for $\Ncal$ and $F$ is as in Definition \ref{def:proj}, then ${\rm Re}(\vgot F):\Ncal\to\r^3$ is a proper conformal minimal immersion in $\r^3$ which properly projects into $\r^2,$ for all $\vgot\in\Zgot.$ 

One can easily check that if $\Zgot\subset\s^1$ is a projector set for $\Ncal,$ then so are $\vgot\Zgot$ for all $\vgot\in\s^1,$ $\Zgot\cup(-\Zgot),$ and any closed subset of $\Zgot.$ 

Pirola's results \cite{pirola} imply that $\s^1$ is a projector set for any parabolic Riemann surface of finite topology (see also \cite{Lo}). 
On the other hand, $\{1\}$ is a universal projector set \cite{al1}, whereas Proposition \ref{pro:sharp} in this paper shows that $\s^1$ is not. In this line we have obtained the following
\begin{quote}
{\bf Main Theorem.} {\em For any finite subset $\Zgot\subset\s^1$ and any open Riemann surface $\Ncal,$ there exists an infinite projector set $\Zgot_\Ncal$ for $\Ncal$ containing $\Zgot.$ 

In particular, $\Zgot$ is a universal projector set.}
\end{quote}

As a corollary, for any open Riemann surface $\Ncal$ and finite set $\Zgot\subset\s^1,$ 
there exist an infinite subset $\Zgot_\Ncal \subset\s^1$ containing $\Zgot$ and a flux-vanishing conformal minimal immersion
$X:\Ncal\to\r^3$ such that $X_\vgot$ properly projects into $\r^2$ for all $\vgot\in\Zgot_\Ncal.$ This particularly answers {\sf (Q1)} in the positive and enlightens about {\sf (Q2)}.

It is not hard to check that if $\s^1$ is a projector set for an open Riemann surface $\Ncal,$ then $\Ncal$ is parabolic (see Proposition \ref{pro:sharp}). Furthermore, if $\Ncal$ is of finite topology and $F:\Ncal \to \c^3$ is a null holomorphic curve such that the map $\Ygot:\s^1\times \Ncal\to\r^2,$ $\Ygot(\vgot,P)={\rm Re}\big(\vgot (F_1,F_2)(P)\big),$ is proper, then $F$ has finite total curvature (see Corollary \ref{co:sharp}). Connecting with a classical Sullivan's conjecture for properly immersed minimal surfaces in $\r^3,$ see \cite{Mo}, the following question remains open:
{\em
\begin{enumerate}[\sf ({Q}1)]
\item[{\sf (Q3)}] Let $\Ncal$ be an open Riemann surface of finite topology, and assume there exists a null holomorphic curve $F:\Ncal\to\c^3$ such that the map
\[
\s^1\times\Ncal \to \r^3,\quad(\vgot,P)\mapsto {\rm Re}\big(\vgot F(P)\big),
\]
is proper. Must $\Ncal$ be of parabolic conformal type? Even more, must $F$ be of finite total curvature?
\end{enumerate}
}


Our main tools come from approximation results for minimal surfaces and null holomorphic curves developed by the authors in \cite{al1,al2}.


\section{Preliminaries}

Denote by $\|\cdot\|$ the Euclidean norm in
$\k^n,$ where $\k=\r$ or $\c.$ For any compact topological
space $K$ and continuous map $f:K \to \k^n,$ denote by
$$\|f\|_{0,K}=\max\{\|f(p)\|\,|\, p \in K\}$$ the maximum norm of $f$ on $K.$

Given an $n$-dimensional topological manifold $M,$ we denote by $\partial M$  the $(n-1)$-dimensional topological manifold determined by its boundary points. For any  $A \subset M,$ $A^\circ$ and $\overline{A}$ will denote the interior and the closure of $A$ in $M,$ respectively.  Open connected subsets of $M-\partial M$ will be called {\em domains}, and those proper topological subspaces of $M$ being $n$-dimensional manifolds with boundary are said to be  {\em regions}. If $M$ is a topological surface,  $M$ is said to be {\em open} if it is non-compact and $\partial M =\emptyset.$

\subsection{Riemann surfaces}

\begin{remark}\label{rem:Nfija}
Throughout this paper $\Ncal$ will denote a fixed but arbitrary open Riemann surface, and $\sigma_\Ncal^2$ a conformal Riemannian metric on $\Ncal.$
\end{remark}

The key tool in this paper is a Mergelyan's type approximation result by null holomorphic curves in $\c^3$ (see Lemma \ref{lem:runge} below and \cite{al1,al2}). This subsection and the next one are devoted to introduce the necessary notations for a good understanding of this result. 

A Jordan arc in $\Ncal$ is said to be analytical if it is contained in an open analytical Jordan arc in $\Ncal.$

Classically, a compact region $A\subset\Ncal$ is said to be Runge if $\Ncal-A$ has no relatively compact components in $\Ncal,$ or equivalently, if the inclusion map ${\rm i}_A: A\hookrightarrow \Ncal$ induces a group monomorphism  $({\rm i}_A)_*:\Hcal_1(A,\z) \to \Hcal_1(\Ncal,\z),$ where $\Hcal_1(\cdot,\z)$ means first homology group with integer coefficients. More generally, an arbitrary subset $A\subset\Ncal$ is said to be {\em Runge} if  $({\rm i}_A)_*:\Hcal_1(A,\z) \to \Hcal_1(\Ncal,\z)$ is injective. In this case we identify the groups  $\Hcal_1(A,\z)$ and  $({\rm i}_A)_*(\Hcal_1(A,\z)) \subset \Hcal_1(\Ncal,\z)$  via $({\rm i}_A)_*$ and consider $\Hcal_1(A,\z) \subset \Hcal_1(\Ncal,\z).$ 


Given an open subset $W\subset \Ncal,$ we denote by
\begin{itemize}
\item $\Fcal_\hgot(W)$ the space of holomorphic functions on $W,$ and
\item $\Omega_\hgot(W)$ the space of holomorphic 1-forms on $W.$
\end{itemize}

The following definition is crucial in our arguments, see Figure \ref{fig:admi}.

\begin{definition}
A compact subset $S\subset\Ncal$ is said to be admissible if and only if
\begin{itemize}
\item $M_S:=\overline{S^\circ}$ is a finite collection of pairwise disjoint compact regions in $\Ncal$ with   $\mathcal{ C}^0$ boundary,
\item $C_S:=\overline{S-M_S}$ consists of a finite collection of pairwise disjoint analytical Jordan arcs, 
\item any component $\alpha$ of $C_S$  with an endpoint  $P\in  M_S$ admits an analytical extension $\beta$ in $\Ncal$ such that the unique component of $\beta-\alpha$ with endpoint $P$ lies in $M_S,$ and 
\item $S$ is Runge.
\end{itemize}
\end{definition}
\begin{figure}[ht]
    \begin{center}
    \scalebox{0.17}{\includegraphics{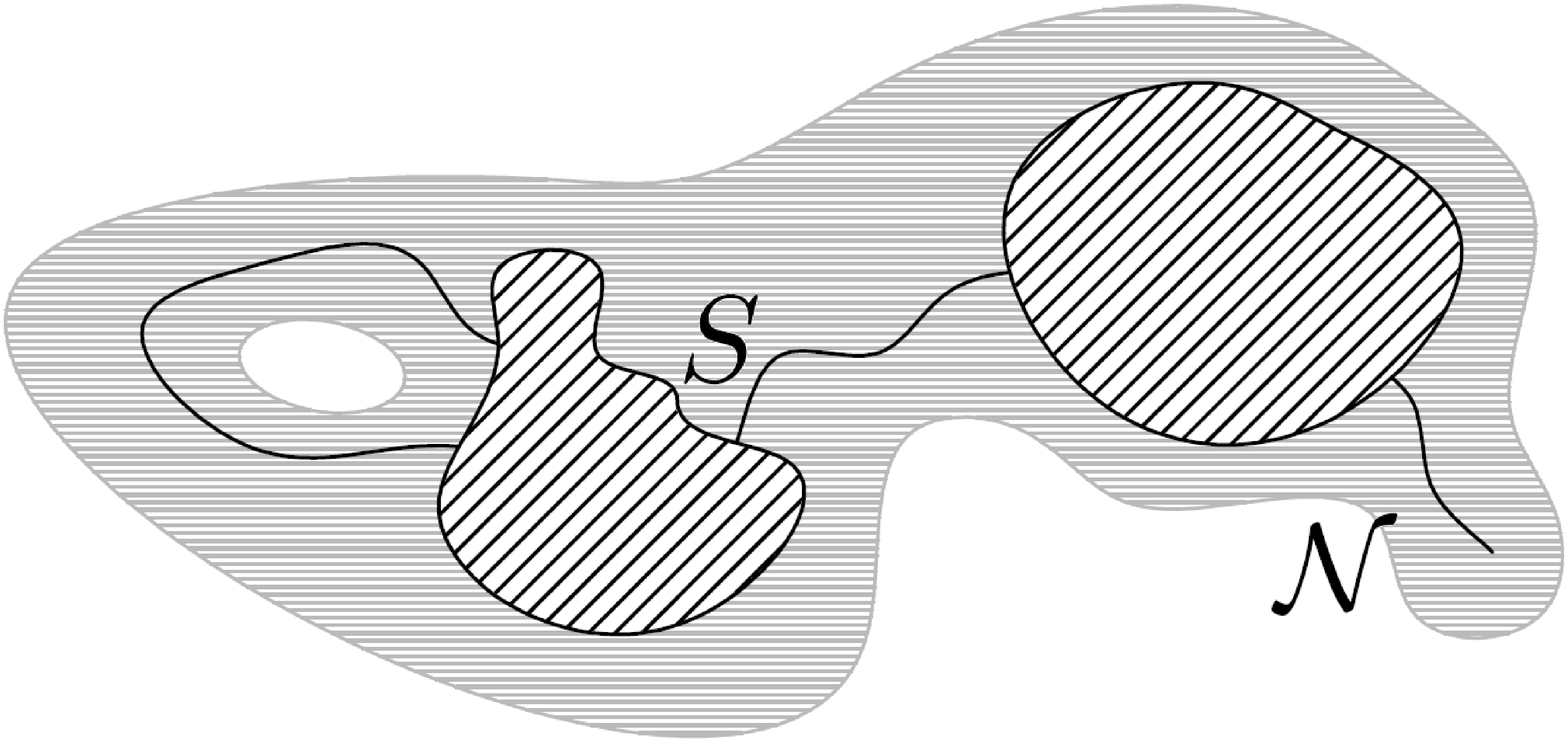}}
        \end{center}
        \vspace{-0.25cm}
\caption{An admissible subset $S$ of $\Ncal$}\label{fig:admi}
\end{figure}



Let $S$ be an admissible subset of $\Ncal.$

A (complex) 1-form $\theta$ on $S$ is said to be of type $(1,0)$ if for any conformal chart $(U,z)$ in $\Ncal,$ $\theta|_{U \cap S}=h(z) dz$ for some function $h:U \cap S \to {\c}.$ An $n$-tuple $\Lambda=(\theta_1,\ldots,\theta_n),$ where $\theta_j$ is a $(1,0)$-type 1-form for all $j,$ is said to be an $n$-dimensional vectorial $(1,0)$-form on $S.$ The space of continuous $n$-dimensional (1,0)-forms on $S$ will be endowed with the $\Ccal^0$ topology induced by the norm 
\begin{equation} \label{eq:norma0}
\|\Lambda\|_{0,S}:=\|\frac{\Lambda}{\sigma_\Ncal}\|_{0,S}=\max_{S} \big\{ \big(\sum_{j=1}^n |\frac{\theta_j}{\sigma_\Ncal}|^2\big)^{1/2}\big\}
\end{equation}
(see Remark \ref{rem:Nfija}).

We denote by  
\begin{itemize}
\item $\Fcal_\hgot(S)$ the space of continuous functions $f:S \to{\c}$ which are holomorphic on an open neighborhood  of $M_S$ in $\Ncal,$ and
\item $\Omega_\hgot(S)$ the space of 1-forms $\theta$ of type $(1,0)$ on $S$ such that $\theta/\vartheta\in \Fcal_\hgot(S)$ for any nowhere-vanishing holomorphic 1-form $\vartheta$ on $\Ncal$ (the existence of such a $\theta$ is well known, see for instance \cite{AFL}). 
\end{itemize}

Smoothness of functions and 1-forms on admissible sets is defined as follows:
\begin{itemize}
\item A function $f \in \Fcal_\hgot(S)$ is said to be smooth if $f|_{M_S}$ admits a smooth extension $f_0$ to a domain $W\subset \Ncal$ containing $M_S,$ and for any component $\alpha$ of $C_S$ and any open analytical Jordan arc $\beta$ in $\Ncal$ containing $\alpha,$  $f$ admits a smooth extension $f_\beta$ to $\beta$
satisfying that $f_\beta|_{W \cap \beta}=f_0|_{W \cap \beta}.$

\item A 1-form $\theta\in \Omega_\hgot(S)$ is said to be smooth if $\theta/\vartheta\in \Fcal_\hgot(S)$ is smooth, for any nowhere-vanishing holomorphic 1-form $\vartheta$ on $\Ncal.$
\end{itemize}

Given a smooth function $f\in \Fcal_\hgot(S),$ the differential $df$ of $f$ is given by 
\[
\text{$df|_{M_S}=d (f|_{M_S})$ and $df|_{\alpha \cap U}=(f \circ \alpha)'(x)dz|_{\alpha \cap U},$}
\]
where $(U,z=x+\imath y),$ $\imath=\sqrt{-1},$ is a conformal chart on $\Ncal$ such that $\alpha \cap U=z^{-1}(\R \cap z(U)).$ Notice that  $df \in \Omega_\hgot(S)$ and is smooth as well.


Finally, the $\mathcal{C}^1$-norm on $S$ of a smooth $f\in \Fcal_\hgot(S)$ is defined by
\[
\|f\|_{1,S}=\max\big\{\|f(P)\|+\|\frac{df}{\sigma_\Ncal}(P)\|\,\big| \;P \in S\big\}.
\]


In a similar way, one can define the notions of smoothness, (vectorial) differential and  $\Ccal^1$-norm for functions $f:S \to \c^k,$ $k \in \n.$


\subsection{Null curves in $\c^3$} \label{sec:null}

Throughout this paper we adopt column notation for both vectors and matrices of linear transformations in $\c^3.$ As usual, $(\cdot)^T$ means transpose matrix. The following operators are strongly related to the
geometry of $\c^3$ and null curves. We denote by
\begin{itemize}
\item $\doble{\cdot,\cdot}:\c^3\times\c^3\to \c,$ $\doble{ u,v }=\bar{u}^T \cdot v,$ the usual Hermitian inner product of $\c^3,$  
\item $\langle\cdot,\cdot\rangle={\rm Re}(\doble{\cdot,\cdot}):\c^3\times\c^3\to\r,$ the Euclidean scalar product of $\c^3\equiv\r^6,$ and 
\item $\curvo{\cdot,\cdot}:\c^3\times\c^3\to\c,$ the complex symmetric bilinear 1-form given by $\curvo{  u, v} = u^T\cdot v.$
\end{itemize}
We also set $\doble{V}^\bot=\{v\in\c^3\,|\, \doble{ u,v}=0 \, \forall u \in V\},$ $\langle V\rangle^\bot=\{v\in\c^3\,|\, \langle u,v \rangle=0 \, \forall u \in V\}$ and  $\curvo{V}^\bot=\{v\in\c^3\,|\,\curvo{  u, v}=0\, \forall u \in V\},$ for any $V\subset \c^3.$ Notice that $\curvo{ \overline{u}}^\bot= \doble{ u}^\bot\subset\esca{u}^\bot$ for all $u\in\c^3,$ and the equality holds if and only if $u=\vec{0}:=(0,0,0)^T.$

A basis $\{u_1,u_2,u_3\}$ of $\c^3$ is said to be {\em $\curvo{\cdot,\cdot}$-conjugate} if $\curvo{ u_j,u_k}=\delta_{jk},$ $j,k\in\{1,2,3\}.$ Likewise, we define the notion of $\curvo{\cdot,\cdot}$-conjugate basis of a complex subspace $U,$ provided that $\curvo{\cdot,\cdot}|_{U \times U}$ is a non-degenerate complex bilinear form.

We denote by $\mathcal{O}(3,\c)$ the complex orthogonal group $\{A\in\mathcal{M}_3(\c)\,|\, A^T A=I_3\},$ that is to say: the group of matrices  whose column vectors determine a $\curvo{\cdot,\cdot}$-conjugate basis of $\c^3.$ We also denote by $A:\c^3 \to \c^3$ the complex linear transformation induced by $A \in \mathcal{O}(3,\c).$ Observe that
\begin{equation}\label{eq:ATheta}
\curvo{Au,Av}=\curvo{u,v} \quad\text{and}\quad \doble{A u,\overline{A} v}=\doble{u,v}, \quad\forall u,v\in\c^3,\; A\in\mathcal{O}(3,\c).
\end{equation}

A vector $u \in \c^3-\{\vec{0}\}$ is said to be null if $\curvo{ u,u}=0.$ We denote by 
\[
\Theta=\{u \in \c^3-\{\vec{0}\} \;|\; u \;\mbox{is null}\}.
\]


Let $M$ be an open Riemann surface. 
\begin{definition}\label{def:null}
A holomorphic map $F:M\to\c^3$ is said to be a null curve if $\curvo{dF,dF}=0$ and $\doble{dF,dF}$ never vanishes on $M.$
\end{definition}
Conversely, given an  exact holomorphic vectorial 1-form $\Phi$ on $M$ satisfying that $\curvo{ \Phi,\Phi}=0$ and $\doble{\Phi,\Phi}$ never vanishes on $M,$ then the map $F:M\to\c^3,$ $F(P)=\int^P \Phi,$ defines a null curve in $\c^3.$ In this case $\Phi=dF$ is said to be the Weierstrass representation of $F.$

A null curve $F:M \to \c^3$ is said to be {\em non-flat} if  $F(M)$ is contained in no null complex straight line. 

%


\begin{definition}
Given a proper subset $M\subset \Ncal,$ we denote by ${\sf N}(M)$ the space of maps $F:M \to\c^3$ extending as a null curve to an open neighborhood of $M$ in $\Ncal.$
\end{definition}

The following definition deals with the notion for null curve on  admissible subsets.

\begin{definition}\label{def:gen-null}
Let $S\subset \Ncal$ be an admissible subset.
A smooth map $F\in\Fcal_\hgot(S)^3$ is said to be a generalized null curve in $\c^3$ if it satisfies the following properties:
\begin{itemize}
\item $F|_{M_S}\in {\sf N}(M_S),$
\item $\curvo{ dF,dF}=0$ and  $\doble{ dF,dF }$ never vanishes on $S.$
\end{itemize}
\end{definition}

If $F$ is a null curve and $A\in\mathcal{O}(3,\c),$ then $A\circ F$ is a null curve as well. The same holds for generalized null curves.

The following Mergelyan's type result for null curves is a key tool in this paper. It will be used to approximate generalized null curves by null curves which are defined on larger domains.

\begin{lemma}[\cite{al1,al2}]\label{lem:runge}
Let $S\subset \Ncal$ be admissible and connected, let $F=(F_j)_{j=1,2,3}\in \Fcal_\hgot(S)^3$ be a generalized null curve in $\c^3,$ and let $W\subset \Ncal$ be a domain of finite topology containing $S$ such that $({\rm i}_S)_*:\mathcal{H}_1(S,\z)\to \mathcal{H}_1(W,\z)$ is an isomorphism, where ${\rm i}_S:S\to W$ denotes the inclusion map. 

Then, for any $\epsilon>0,$ there exists $H=(H_j)_{j=1,2,3}\in{\sf N}(W)$ such that $\|H-F\|_{1,S}<\epsilon.$
Moreover, one can choose $H_3=F_3$ provided that $F_3\in\Fcal_\hgot(W)$ and $dF_3$ never vanishes on $C_S.$ 
\end{lemma}


\section{Main Lemma}

Let us start by introducing some notation.

Let $\Zgot=\{\vgot_1,\ldots,\vgot_\ngot\} \subset \s^1$ with
cardinal number $\ngot\in \n.$
Let $u=(z_1,z_2,z_3)\in\c^3,$ let $\vgot \in \Zgot,$ let $X$ be a topological space, let $K \subset X$ be a compact subset,  and let $F=(F_1,F_2,F_3):X\to \c^3$ be a continuous map. We denote by
\begin{itemize}
\item $u^\vgot={\rm Re}[\vgot (z_1,z_2)]\in \r^2,$
\item $F^\vgot={\rm Re}[\vgot (F_1,F_2)]:X \to \r^2,$ and $F^\Zgot=(F^{\vgot_j})_{j=1,\ldots,\ngot}:X \to \r^{2 \ngot},$
\item $\mini{F^\Zgot}:X \to \r,$ $\mini{F^\Zgot}(P)=\min\{\|F^{\vgot}(P)\|\,|\; \vgot \in \Zgot\},$ and  $\mini{F^\Zgot}_K=\min_K \mini{F^\Zgot}.$
\end{itemize}

The following technical result is the core of our construction. 

\begin{lemma}\label{lem:fun}
Let $M,$ $V$ be two admissible compact regions in $\Ncal$ such that $M\subset V^\circ.$ Let $\Zgot\subset\s^1$  be a finite subset with cardinal number $\ngot,$  and consider $F\in{\sf N}(M)$ and $\delta>0$ such that
\begin{equation}\label{eq:lema}
\mini{F^\Zgot}_{\partial M}>\delta.
\end{equation}

Then, for any $\epsilon>0$ and any $\kappa>\delta,$ there exists $\hat{F}\in{\sf N}(V)$ satisfying
\begin{enumerate}[{\sf ({L}1)}]
\item $\|\hat{F}-F\|_{1,M}<\epsilon,$
\item $\mini{\hat{F}^\Zgot}_{V-M^\circ}>\delta/\ngot,$ and
\item $\mini{\hat{F}^\Zgot}_{\partial V}>\kappa.$
\end{enumerate}
\end{lemma}

Roughly speaking, the lemma asserts that a finite family of compact associated minimal surfaces whose boundaries lie outside a cylinder in $\r^3$ can be stretched
near the boundary, in such a way that the boundaries of the new associated surfaces lie outside a larger parallel cylinder. In this process the topology and even the conformal structure of the arising family can be chosen arbitrarily large. See Figure \ref{fig:lemma}.
\begin{figure}[ht]
    \begin{center}
    \scalebox{0.4}{\includegraphics{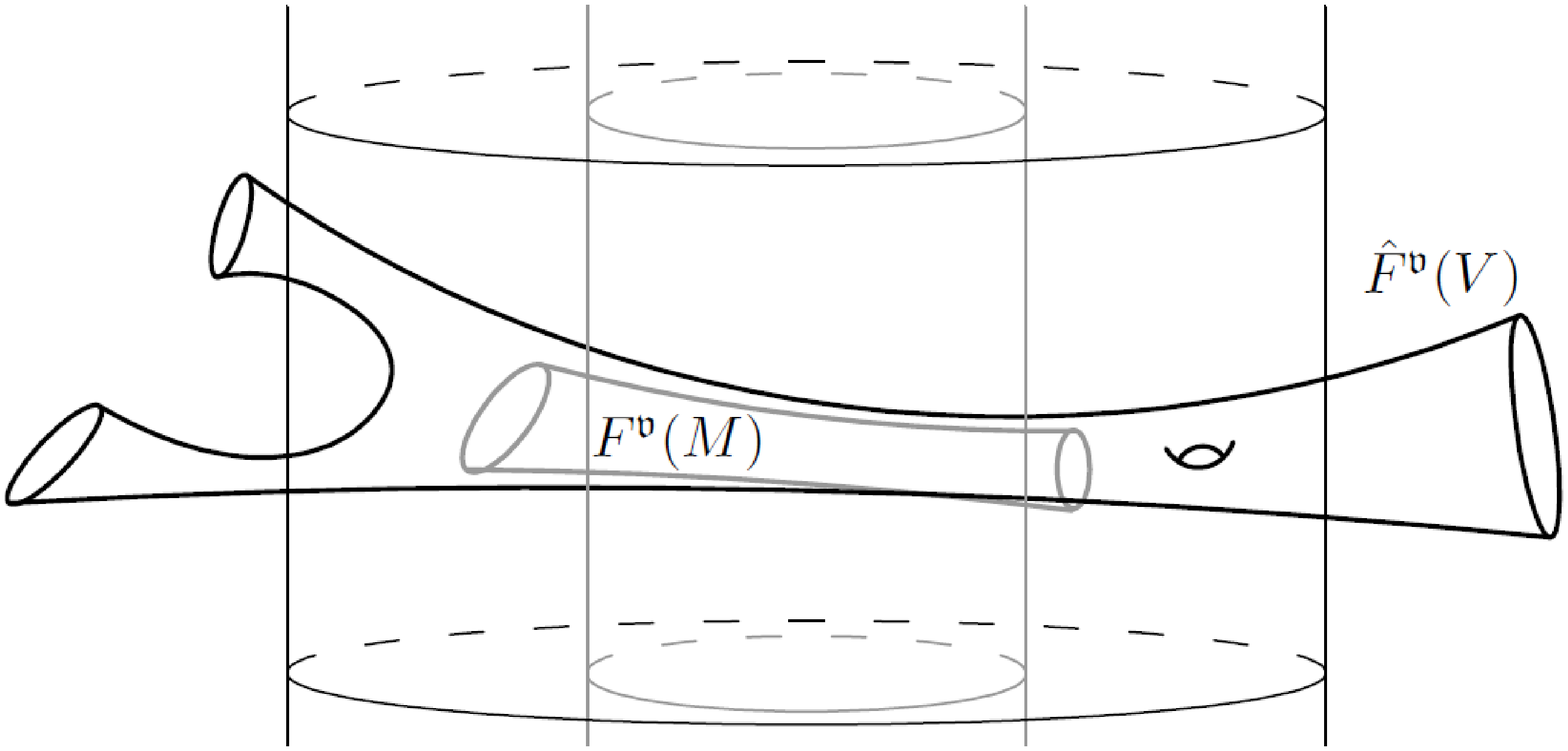}}
        \end{center}
        \vspace{-0.25cm}
\caption{Lemma \ref{lem:fun}}\label{fig:lemma}
\end{figure}

The proof of Lemma \ref{lem:fun} goes by induction on (minus) the Euler characteristic of $V-M^\circ$ (notice that $-\chi(V -M^\circ)\geq 0$). 
The hard part of the proof is the basis of the induction, which roughly goes as follows. 

Firstly we split $\partial M$ into a suitable family of small Jordan arcs $\alpha_{i,j}$ (see properties {\sf (a1)}, {\sf (a2)}, and {\sf (a3)} below), and assign to each of them a complex direction $e_{i,j}$ in $\c^3$ (see \eqref{eq:normal}). The splitting is made so that deformations of $F$ around $\alpha_{i,j}$ preserving the direction $e_{i,j},$ keep the boundaries of all the $\Zgot$-associated minimal surfaces outside the cylinder of radius $\delta/\ngot.$ This choice is possible by basic trigonometry, see Claim \ref{cla:bolita}.

In a second step, we construct an admissible set $S$ by  attaching to $M$ a family of Jordan arcs $r_{i,j}$ connecting $\alpha_{i,j}$ and $\partial V.$ Then, we approximate $F$ on $M$ by a null curve $H \in {\sf N}(V)$ formally satisfying the theses of the lemma on $S,$ see items {\sf (c1)} to {\sf (c4)}. 

Finally, we modify $H$ hardly on $S$ and strongly on $V-S$ in a recursive way to obtain the null curve $\hat{F}\in{\sf N}(V)$ which proves the basis of the induction, see Claim \ref{cla:recur}. This deformation pushes the boundaries of the $\Zgot$-associated surfaces of $H(V)$ outside the cylinder of radius $\kappa.$ Furthermore, this process hardly modifies the  $e_{i,j}$-coordinate of $H$ on the connected component $\Omega_{i,j}$ of $V-S$ with $\alpha_{i,j}\subset\partial\Omega_{i,j},$ see {\sf (f2)}. Therefore, the $\Zgot$-associated surfaces of the arising null curve $\hat{F}(V-M^\circ)$  lie outside the cylinder of radius $\delta/\ngot.$

For the inductive step we reason as follows. If $-\chi(V-M^\circ)=n\in\n,$ we use Lemma \ref{lem:runge} as a bridge principle for null curves to obtain a region $U$ with $M\subset U^\circ\subset U\subset V^\circ$ and $-\chi(V-U^\circ)=n-1,$ and a null curve $H\in{\sf N}(U)$ which approximates $F$ on $M$ and satisfies $\mini{H^\Zgot}_{\partial M}>\delta.$ Then, we finish by applying the induction hyphotesis.

\subsection{Basis of the induction}

Let us show that Lemma \ref{lem:fun} holds in the particular instance $\chi(V-M^\circ)=0.$

Up to slightly deforming $F$ (use Lemma \ref{lem:runge}), we can suppose that $F$ is non-flat.

Since $M \subset V^\circ$ and $V^\circ-M$ has no bounded components in $V^\circ,$ then
 $V-M^\circ=\cup_{i=1}^k \Acal_i,$ where $\Acal_1,\ldots, \Acal_k$ are pairwise disjoint compact annuli. Write $\partial\Acal_i=\alpha_i \cup \beta_i,$ where $\alpha_i\subset \partial M$ and $\beta_i\subset \partial V$  for all $i.$
 
Denote by $B(r)$ the {\em $2$-dimensional}  Euclidean ball $\{p\in \r^2\,|\, \|p\|<r\}$ for any $r>0.$

Label $\Delta=(\r^2-\overline{B(\delta)})^\ngot \subset \r^{2 \ngot},$ and for any $x\in\Delta$ choose a vectorial line $L_x\subset\r^2$  and an open neighborhood $U_x$ of $x$ in $\Delta$ such that
\begin{equation}\label{eq:rectas}
(q_j+L_x)\cap \overline{B(\delta/\ngot)}=\emptyset,\quad \forall (q_1,\ldots,q_\ngot)\in U_x,
\end{equation}
see Figure \ref{fig:rectabuena}.
\begin{figure}[ht]
    \begin{center}
    \scalebox{0.4}{\includegraphics{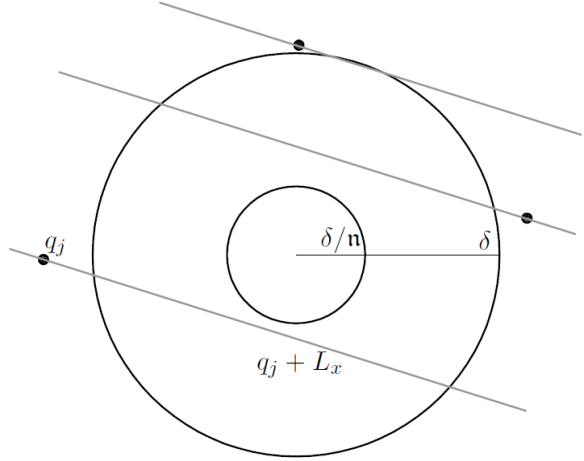}}
        \end{center}
        \vspace{-0.25cm}
\caption{The vectorial line $L_x\subset\r^2$}\label{fig:rectabuena}
\end{figure}
The existence of $L_x$ and $U_x,$ $x \in \Delta,$  follows straightforwardly from the following
\begin{claim}\label{cla:bolita}
For any $x_1,\ldots,x_\ngot \in \r^2-\overline{B(1)},$ there exists a vectorial line $L\subset \r^2$ depending on $x_1,\ldots,x_\ngot,$ such that
$$(x_j+L) \cap \overline{B(1/\ngot)}=\emptyset, \quad j=1,\ldots,\ngot.$$
\end{claim}
\begin{proof}
Label $W_j=\{e^{\imath t} \frac{x_j}{\|x_j\|}\,|\; t \in ]-\frac{\pi}{2\ngot},\frac{\pi}{2 \ngot}[\}$ for all $j=1,\ldots,\ngot,$ and take $x_0 \in \s^1-(\cup_{j=1}^\ngot W_j).$ Setting $L=\{t x_0\,|\; t \in \r\},$ elementary trigonometry gives that $( \frac{x_j}{\|x_j\|}+L) \cap B(\sin(\frac{\pi}{2\ngot}))=\emptyset,$ $j=1,\ldots,\ngot.$ Since $\|x_j\|>1$ for all $j$ and $\sin(\frac{\pi}{2\ngot}) > \frac1{\ngot},$ we are done.
\end{proof}

For each $n\in\n$ denote by $\z_n=\{0,\ldots,n-1\}$ the additive cyclic group of integers modulus $n.$ Since $\mathcal{U}:=\{U_x\,|\, x\in\Delta\}$ is an open covering of the compact set $F^\Zgot(\partial M)\subset\Delta$ (see \eqref{eq:lema}), there exist $m \in \n,$ $m \geq 3,$  and a collection $\{\alpha_{i,j} \,|\, (i,j)\in \{1,\ldots,k\}\times \z_m\}$ such that for each $i \in \{1,\ldots,k\},$

\begin{enumerate}[{\sf ({a}1)}]
\item $\cup_{j=1}^{m} \alpha_{i,j}= \alpha_i,$
\item $\alpha_{i,j}$ and $\alpha_{i,j+1}$ have a common endpoint $Q_{i,j}$ and are otherwise disjoint for all  $j\in \z_m,$ and
\item there exists $U_{i,j}\in \mathcal{U}$ such that $F^\Zgot(\alpha_{i,j}) \subset U_{{i,j}},$  for all  $j\in \z_m.$
\end{enumerate}

If $U_{i,j}=U_{x_{i,j}}$ for $x_{i,j}\in\Delta,$ for simplicity we write $L_{i,j}=L_{x_{i,j}}$ for all $(i,j)\in \{1,\ldots,k\}\times \z_m.$

Let $\{r_{i,j}\,|\;j\in\z_m\}$ be a collection of pairwise
disjoint analytical Jordan arcs in $\Acal_i$ such that $r_{i,j}$ has
initial point $Q_{i,j} \in \alpha_i,$ final point $P_{i,j}\in \beta_i,$ and $r_{i,j}$ is otherwise disjoint from $\partial\Acal_i$ for all $i$ and $j.$ Without loss of generality, assume that $S=M\cup(\cup_{i,j} r_{i,j})$ is admissible. See Figure \ref{fig:anillo}.
\begin{figure}[ht]
    \begin{center}
    \scalebox{0.3}{\includegraphics{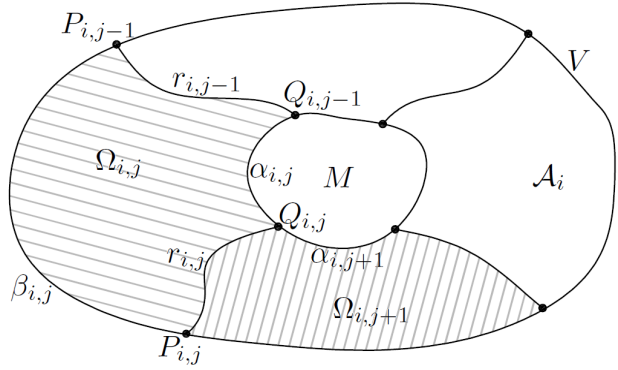}}
        \end{center}
        \vspace{-0.25cm}
\caption{The annulus $\Acal_i\subset V-M^\circ$}\label{fig:anillo}
\end{figure}

The first deformation stage starts with the following

\begin{claim}\label{cla:step1} There exists a generalized null curve $G:S\to\c^3$ such that
\begin{enumerate}[{\sf ({b}1)}]
\item $G|_{M}=F,$
\item $(p+L_{i,h})\cap \overline{B(\delta/\ngot)}=\emptyset,$ $\forall p\in \cup_{\vgot \in \Zgot} G^{\vgot}(r_{i,j}),$ $\forall h\in\{j,j+1\},$ $\forall i,j,$ and
\item $(G^{\vgot}(P_{i,j})+L_{i,h})\cap \overline{B(\kappa)}=\emptyset,$ $\forall \vgot \in \Zgot,$ $\forall h\in\{j,j+1\},$ $\forall i,j.$
\end{enumerate}
\end{claim}
\begin{proof}
Write $F=(F_1,F_2,F_3)$ and for each $i,j$ set $$d_{i,j}(t)=\left(t F_1(Q_{i,j}),t F_2(Q_{i,j}),F_3(Q_{i,j})+(t-1) \imath \sqrt{F_1(Q_{i,j})^2+F_2(Q_{i,j})^2}\right), \; t \geq 1.$$ From \eqref{eq:lema} one has $(|F_1|+|F_2|)(Q_{i,j})\neq 0,$ and so $d_{i,j}$ is a real half-line in $\c^3$ satisfying $\curvo{d_{i,j}'(t),d_{i,j}'(t)}=0.$ Moreover, $d_{i,j}(t)^{\vgot}=t F^{\vgot}(Q_{i,j})$ for all $t \geq 1$ and $\vgot \in \Zgot.$
Since $(F^{\vgot}(Q_{i,j})+L_{i,h})\cap\overline{B(\delta/\ngot)}=\emptyset$ (see \eqref{eq:rectas} and (a3)), then the vector $F^{\vgot}(Q_{i,j})$ points to the connected component of  $\r^2-(F^{\vgot}(Q_{i,j})+L_{i,h})$ disjoint from $\overline{B(\delta/\ngot)},$ and so $(d_{i,j}(t)^{\vgot}+L_{i,h})\cap \overline{B(\delta/\ngot)}=\emptyset$ for all $t\geq 1,$ $h\in \{j, j+1\}$ and $\vgot \in \Zgot.$ Furthermore, we can take $t_0>1$ so that $(d_{i,j}(t_0)^{\vgot}+L_{i,h})\cap \overline{B(\kappa)}=\emptyset$ for all $\vgot \in \Zgot$ and $ h\in \{j,j+1\}.$  Up to a slightly smoothing around the points $Q_{i,j}$ for all $i,j,$ it suffices to set $G(r_{i,j})=d_{i,j}([1,t_0])$ for all $i,j$ and $G|_M=F.$
\end{proof}

Then Lemma \ref{lem:runge} applied to $G$ straightforwardly provides a non-flat $H\in{\sf N}(V)$ satisfying that
\begin{enumerate}[{\sf ({c}1)}]
\item $\|H-F\|_{1,M}<\epsilon/(km+1),$
\item $(p+L_{i,h})\cap \overline{B(\delta/\ngot)}=\emptyset,$ $\forall p\in \cup_{\vgot \in \Zgot} H^{\vgot}(r_{i,j}),$ $\forall h\in\{j,j+1\},$ $\forall i,j,$
\item $(H^{\vgot}(P_{i,j})+L_{i,h})\cap \overline{B(\kappa)}=\emptyset,$ $\forall \vgot \in \Zgot,$ $\forall h\in\{j,j+1\},$ $\forall i,j,$ and
\item  $H^\Zgot(\alpha_{i,j})  \subset U_{{i,j}}\in \mathcal{U}$ for all  $i$ and $j.$
\end{enumerate}

Roughly speaking, properties {\sf (c1)}, {\sf (c2)}, and {\sf (c3)} mean that $H$ satisfies {\sf (L1)}, {\sf (L2)}, and {\sf (L3)} just on $S,$ respectively.

Denote by $\Omega_{i,j}$ the closed disc in $\Acal_i$
bounded by $\alpha_{i,j}\cup r_{i,j-1} \cup r_{i,j}$ and a piece,
named $\beta_{i,j},$ of $\beta_i$ connecting $P_{i,j-1}$ and
$P_{i,j}.$ Obviously $\Omega_{i,j} \cap \Omega_{i,j+1}=r_{i,j}$ $\forall i,j,$ and
$\Acal_i=\cup_{j=1}^{m} \Omega_{i,j}$ for all $i.$ See Figure \ref{fig:anillo}.

Let $\eta:\{1,\ldots,k m\}\to\{1,\ldots,k\}\times \z_m$ be the bijection $\eta(n)=(\mathcal{E}(\frac{n-1}{m})+1,n-1),$ where $\mathcal{E}(\cdot)$ means integer part.

The second deformation process is included in the following

\begin{claim}\label{cla:recur}
There exists a sequence $H_0=H,H_1,\ldots,H_{km}$ of non-flat null curves in ${\sf N}(V)$ such that
\begin{enumerate}[{\sf ({d}1{$_n$})}]
\item $\|H_n-H_{n-1}\|_{1,\overline{V-\Omega_{\eta(n)}}}<\epsilon/(km+1),$ $n\geq 1,$

\item $(p+L_h)\cap \overline{B(\delta/\ngot)}=\emptyset,$ $\forall p\in \cup_{\vgot \in \Zgot} {H_n^\vgot}(r_{\eta(a)}),$ $\forall h\in\{\eta(a),\eta(a)+(0,1)\},$ $\forall a\in\{1,\ldots,km\},$

\item $(H_n^{\vgot}(P_{\eta(a)})+L_h)\cap \overline{B(\kappa)}=\emptyset,$ $\forall \vgot \in \Zgot,$ $\forall h\in\{\eta(a),\eta(a)+(0,1)\},$ $\forall a\in\{1,\ldots,km\},$

\item $H_n^\Zgot(\alpha_{\eta(a)})  \subset U_{\eta(a)},$ $\forall a\in\{1,\ldots,km\},$

\item $\mini{H_n^\Zgot}_{\Omega_{\eta(a)}}>\delta/\ngot$ for all $a\in\{1,\ldots, n\},$ $n\geq 1,$ and

\item $\mini{H_n^\Zgot}_{\beta_{\eta(a)}}>\kappa$ for all $a\in\{1,\ldots, n\},$ $n\geq 1.$
\end{enumerate}
\end{claim}

Now, properties {\sf (d1$_n$)}, {\sf (d5$_n$)}, and {\sf (d6$_n$)} imply that $H_n$ formally satisfies {\sf (L1)}, {\sf (L2)}, and {\sf (L3)} on $M\cup(\cup_{a=1}^n \Omega_{\eta(a)}).$ In particular, $H_{km}$ will solve Lemma \ref{lem:fun}.

\begin{proof}[Proof of Claim \ref{cla:recur}]
From {\sf (c2)}, {\sf (c3)}, and {\sf (c4)}, $H_0=H$ satisfies {\sf (d2$_0$)}, {\sf (d3$_0$)}, and {\sf (d4$_0$)}, whereas the remaining properties make no sense for $n=0.$
Reason by induction and assume that we already have $H_0,\ldots, H_{n-1},$ $n\geq 1,$ satisfying the corresponding properties. Let us construct $H_n.$

For any  $\nu=(x,y)\in \r^2-\{(0,0)\},$ it is clear that   $e=(x,x,y,y,0,0)\equiv (1+\imath)(x,y,0)$ is a non-null vector in $\c^3$ and $\doble{e}^\bot =\{u\in\c^3 \,|\, \esca{u^{\vgot},\nu}=0\; \forall \vgot\in \s^1\}$ (here $\esca{\cdot,\cdot}$ denotes the escalar product in $\r^2$).

In particular, if $\nu_n=(x_n,y_n)\in\r^2$ is a unit normal vector to $L_{\eta(n)},$ $e_n:=(1+\imath)(x_n,y_n,0),$ and $w_n:=\overline{e_n}/\curvo{\overline{e_n},\overline{e_n}},$ one has 
\begin{equation}\label{eq:normal}
\curvo{w_n}^\bot=\doble{e_n}^\bot \subset \{u\in\c^3 \,|\, \esca{u^{\vgot},\nu_n}=0\; \forall \vgot\in\Zgot\}.
\end{equation}
(The inclusion in equation \eqref{eq:normal} becomes an equality when $\Zgot\nsubseteq \{\vgot,-\vgot\}$ for all $\vgot \in \s^1.$)

Since $w_n$ is not null, we can take $u_n, v_n\in \curvo{w_n}^\bot$ so that $\{u_n,v_n,w_n\}$ is a $\curvo{\cdot,\cdot}$-conjugate basis of $\c^3.$ Consider the complex orthogonal matrix 
$
A_n=(u_n,v_n,w_n)^{-1},
$ 
 define $G_n:=A_n\circ H_{n-1}\in{\sf N}(V),$ and write $G_n=(G_{n,1},G_{n,2},G_{n,3}).$
Notice that
\begin{equation}\label{eq:An}
A_n(\doble{e_n}^\bot)= \{(z_1,z_2,0) \in \c^3\,|\, z_1,\,z_2 \in \c\}.
\end{equation}

Choose a closed disc $K_n$ in $\Omega_{\eta(n)}-(r_{\eta(n)-(0,1)} \cup \alpha_{\eta(n)}\cup r_{\eta(n)})$ such that
\begin{enumerate}[{\sf ({e}1)}]
\item $K_n\cap \beta_{\eta(n)}$ is a Jordan arc,
\item $(p+L_{\eta(n)})\cap \overline{B(\delta/\ngot)}=\emptyset,$ $\forall p\in \cup_{\vgot \in \Zgot} {H_{n-1}^\vgot}(\overline{\Omega_{\eta(n)}-K_n}),$ and
\item $(p+L_{\eta(n)})\cap \overline{B(\kappa)}=\emptyset,$ $\forall p\in \cup_{\vgot \in \Zgot} {H_{n-1}^\vgot}(\overline{\beta_{\eta(n)}-K_n}).$
\end{enumerate}
This choice can be guaranteed by a continuity argument. 
Property {\sf (e2)} follows from {\sf (d2$_{n-1}$)},  {\sf (d4$_{n-1}$)}, and \eqref{eq:rectas}, whereas {\sf (e3)} follows from {\sf (d3$_{n-1}$)}.

Consider now a Jordan arc $\gamma_n\subset \overline{\Omega_{\eta(n)}-K_n}$ with endpoints $R_n\in \alpha_{\eta(n)}-\{Q_{\eta(n)-(0,1)},Q_{\eta(n)}\}$ and $T_n\in (\partial K_n)-\beta_{\eta(n)},$ and otherwise disjoint from $K_n\cup (\partial \Omega_{\eta(n)})$ (see Figure \ref{fig:omega}).
\begin{figure}[ht]
    \begin{center}
    \scalebox{0.3}{\includegraphics{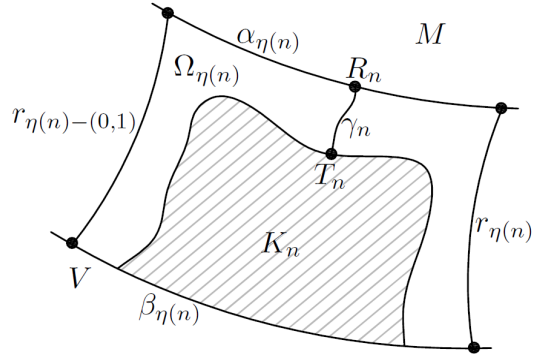}}
        \end{center}
        \vspace{-0.25cm}
\caption{The closed disc $\Omega_{\eta(n)}$}\label{fig:omega}
\end{figure}
Without loss of generality, assume that $K_n$ and $\gamma_n$ are chosen so that the compact set $S_n:=(\overline{V-\Omega_{\eta(n)}})\cup K_n\cup \gamma_n$ is admissible and $(dG_{n,3})|_{\gamma_n}$ never vanishes (recall that $H_{n-1}$ is non-flat and therefore so is $G_n$).

Consider $\vgot_0\in\s^1\subset\c$ such that ${\rm Re}(\vgot\vgot_0)\neq 0$ $\forall \vgot\in\Zgot,$ denote by $\mu=\vgot_0(y_n,-x_n,\imath)\in\c^3,$ and observe that $\mu^\vgot={\rm Re}(\vgot\vgot_0)(y_n,-x_n) \neq (0,0)$ for all $\vgot \in \Zgot.$ Therefore, there exists a large enough $C_n>0$ such that
\begin{equation}\label{eq:fuera}
\mini{(\mu_n+H_{n-1})^\Zgot}_{K_n}>\kappa,
\end{equation}
where $\mu_n=C_n\mu.$ Notice also that $\mu_n\in \Theta\cap\doble{e_n}^\bot.$


Denote by $\zeta=A_n(\mu_n)\in\Theta\cap A_n(\doble{e_n}^\bot).$ Taking into account \eqref{eq:An}, there exists a null vector $\zeta^* \in A_n(\doble{e_n}^\bot)$ so that $\{\zeta,\zeta^*\}$ is a basis of $A_n(\doble{e_n}^\bot)$ and $\curvo{\zeta,\zeta^*} \neq 0.$

Let $\gamma_n(u),$ $u \in [0,1],$ be a smooth parameterization of $\gamma_n$ with $\gamma_n(0)=R_n.$  Label  $\tau_j= \gamma_n([0,1/j])$ and consider the parameterization   $\tau_j(u)=\gamma_n(u/j),$ $u \in [0,1].$ Write $Y_j(u)=G_{n,3}(\tau_j(u)),$  $u \in [0,1],$ and notice that $\frac{d Y_j}{du}(0)=\frac1{j}\frac{d (G_{n,3}\circ \gamma_n)}{du}(0)\neq0 $ for all $j \in \n.$
Set $\zeta_j=\zeta- \frac{(d Y_j/du (0))^2}{2\curvo{\; \zeta,\zeta^*\;}} \zeta^*,$ $j \in \n,$ and observe that $\lim_{j \to \infty} \zeta_j =\zeta$ and  $\curvo{\zeta_j,\zeta_j}=-(\frac{d Y_j}{du}(0))^2\neq 0$ for all $j.$

Set $h_j:[0,1] \to \c^3,$  
\[
h_j(u)=G_n(R_n)+ \imath \frac{Y_j(u)-Y_j(0)}{\curvo{\;\zeta_j,\zeta_j\;}^{1/2}} \zeta_j + (0,0,Y_j(u)-G_{n,3}(R_n)).
\]
Since $\zeta_j\in A_n(\doble{e_n}^\bot),$ then $\curvo{ h_j'(u),h_j'(u) }=0$ and $\doble{ h_j'(u),h_j'(u) }$ never vanishes on $[0,1],$  $j\in \n,$ see \eqref{eq:An}. Up to choosing a suitable branch of $\curvo{ \zeta_j,\zeta_j}^{1/2},$ the sequence  $\{h_j\}_{j \in \n}$ converges uniformly on $[0,1]$ to $h_\infty:[0,1] \to \c^3,$  
\[
h_\infty(u)=u \zeta+ G_n(R_n).
\]

On the other hand $\{\tau_{j}(1)\}_{j\in\n}\to R_n,$ and so $\{h_{j}(1)-G_n(\tau_{j}(1))\}_{j\in\n}\to \zeta=A_n(\mu_n).$ Taking into account \eqref{eq:fuera}, there exists a large enough $j_0\in\n$ such that
\begin{equation}\label{eq:fuera1}
\mini{\left(A_n^{-1}(h_{j_0}(1)-G_n(\tau_{j_0}(1)))+H_{n-1}\right)^\Zgot}_{K_n}>\kappa.
\end{equation}

Set $\hat{h}:\tau_{j_0}([0,1]) \to \c^3,$ $\hat{h}(P)=h_{j_0}(\tau_{j_0}^{-1}(P)).$ Identify $\tau_{j_0}\equiv\tau_{j_0}([0,1]),$ and denote by $\hat{G}_n=(\hat{G}_{n,1},\hat{G}_{n,2},\hat{G}_{n,3}): S_n\to \c^3$ the continuous map given by
\begin{multline}\label{eq:Ggorro}
{\hat{G}_n}|_{\overline{V-\Omega_{\eta(n)}}}={G_n}|_{\overline{V-\Omega_{\eta(n)}}},\; \hat{G}_n|_{\tau_{j_0}}=\hat{h},\\ {\hat{G}_n}|_{(\gamma_n-\tau_{j_0})\cup K_n}={G_n}|_{(\gamma_n-\tau_{j_0})\cup K_n}-G_n(\tau_{j_0}(1))+\hat{h}(\tau_{j_0}(1)).
\end{multline}
Notice that $\hat{G}_{n,3}=(G_{n,3})|_{S_n}.$
The equation $\curvo{d\hat{G}_n,d \hat{G}_n}=0$ formally holds except at the points $R_n$ and $\tau_{j_0}(1)$ where smoothness could fail.
 Up to smooth approximation (only affecting to $\hat{G}_{n,1}$ and $\hat{G}_{n,2}$),  $\hat{G}_n$ is a generalized null curve  satisfying that
\begin{equation}\label{eq:Ggorro1}
{\hat{G}_n}|_{\overline{V-\Omega_{\eta(n)}}}={G_n}|_{\overline{V-\Omega_{\eta(n)}}},\; \hat{G}_{n,3}=G_{n,3}|_{S_n},\; \mini{(A_n^{-1}\circ\,\hat{G}_n)^\Zgot}_{K_n}>\kappa.
\end{equation}
Here we have taken into account \eqref{eq:fuera1}, \eqref{eq:Ggorro}, and $H_{n-1}=A_n^{-1}\circ G_n.$ Applying Lemma \ref{lem:runge} to $\hat{G}_n$ and $S_n$ we can find $Z_n=(Z_{n,1},Z_{n,2},Z_{n,3})\in{\sf N}(V)$ so that
\begin{itemize}
\item $\|Z_n-\hat{G}_n\|_{1,\overline{V-\Omega_{\eta(n)}}}<\epsilon_0,$ where $\epsilon_0>0$ will be specified later,
\item $Z_{n,3}={G}_{n,3}$ on $V,$ and
\item $\mini{(A_n^{-1}\circ Z_n)^\Zgot}_{K_n}>\kappa.$
\end{itemize}

Set $H_n:=A_n^{-1}\circ Z_n\in{\sf N}(V),$ and let us rewrite these properties in terms of $H_n$ and $H_{n-1}$ (recall that $G_n=A_n\circ H_{n-1}$):
\begin{enumerate}[{\sf ({f}1)}]
\item $\|H_n-H_{n-1}\|_{1,\overline{V-\Omega_{\eta(n)}}}<\epsilon_0\cdot \|A_n^{-1}\|$ (see \eqref{eq:Ggorro1}),
\item $\doble{H_n-H_{n-1},e_n}=0$ on $V$ (see \eqref{eq:An}), and
\item $\mini{H_n^\Zgot}_{K_n}>\kappa.$
\end{enumerate}

To finish, let us check that $H_n$ satisfies the required properties provided that $\epsilon_0$ is small enough. Indeed, {\sf (f1)} directly gives {\sf (d1$_n$)} if $\epsilon_0<\frac{\epsilon}{\|A_n^{-1}\|(1+km)}.$ Moreover, {\sf (d2$_n$)} (respectively, {\sf (d3$_n$)}, {\sf (d4$_n$)}) follows from {\sf (f1)} and {\sf (d2$_{n-1}$)} (respectively, {\sf (d3$_{n-1}$)}, {\sf (d4$_{n-1}$)})  for a small enough $\epsilon_0.$

To check {\sf (d5$_n$)} we distinguish two cases. If $a<n$ (and so $n>1$), then we finish by using {\sf (d5$_{n-1}$)} and {\sf (f1)} for a small enough $\epsilon_0.$ In case $a=n$ we argue as follows.  Assume first that $P\in \Omega_{\eta(n)}-K_n.$ Then {\sf (f2)} gives that $H_n(P)-H_{n-1}(P)\in\doble{e_n}^\bot,$ and so, by \eqref{eq:normal}, $\langle H_n^\vgot(P)-H_{n-1}^\vgot(P),\nu_n\rangle=0$ $\forall \vgot \in \Zgot.$ Write
$H_n^\vgot(P)=H_{n-1}^\vgot(P)+(H_n^\vgot(P)-H_{n-1}^\vgot(P)),$ and notice that  $H_{n-1}^\vgot(P)\in H_{n-1}^\vgot(\overline{\Omega_{\eta(n)}-K_n})$ and $H_n^\vgot(P)-H_{n-1}^\vgot(P) \in L_{\eta(n)},$ $\forall\vgot\in\Zgot.$ By {\sf (e2)} we infer that $\mini{H_n^\Zgot}(P)>\delta/\ngot$ and we are done. Assume now that $P\in K_n.$ In this case, {\sf (f3)} directly gives that $\mini{H_n^\Zgot}(P)>\kappa>\delta/\ngot$ as well.

The proof of {\sf (d6$_n$)} is analogous to that of {\sf (d5$_n$)}. In case $a<n,$ we use {\sf (d6$_{n-1}$)} and {\sf (f1)} for small enough $\epsilon_0.$ In case $a=n,$ we argue as in the proof of {\sf (d5$_n$)} but using {\sf (e3)} instead of {\sf (e2)}. In this case we get that $\mini{H_n^\Zgot}(P)>\kappa$ for all $P\in \beta_{\eta(n)}-K_n.$ Finally, {\sf (f3)} shows that $\mini{H_n^\Zgot}(P)>\kappa$ for all $P\in K_n.$

The proof of Claim \ref{cla:recur} is done.
\end{proof}

Set $\hat{F}:=H_{km}\in{\sf N}(V).$ Properties {\sf (c1)} and {\sf (d1$_n$)}, $n=1,\ldots, km,$ imply {\sf (L1)}, whereas properties {\sf (L2)} and {\sf (L3)} directly follow from {\sf (d5$_{km}$)} and {\sf (d6$_{km}$)}, respectively. Therefore, $\hat{F}$ solves the basis of the induction.

\subsection{Inductive step}

Let $n\in\n,$ assume that Lemma \ref{lem:fun} holds when $-\chi(V -M^\circ)< n,$ and let us show that it also holds when
$-\chi(V-M^\circ)=n.$

Recall that $M$ is admissible, and so $\Hcal_1(M,\Z)\subset \Hcal_1(V,\Z).$ Since $-\chi(V-M^\circ)=n>0$, there exists a Jordan curve $\hat{\gamma}\subset V^\circ$ intersecting $V^\circ-M^\circ$ in a Jordan arc $\gamma$ with endpoints $P,Q\in\partial M$ and otherwise disjoint from $\partial M,$ and such that $\hat\gamma\in\Hcal_1(V,\Z)-\Hcal_1(M,\Z).$ Consequently, since $V$ is admissible then $\hat{\gamma}$ can be chosen so that $S:=M\cup\gamma$ is admissible as well.

At this point, we need the following
\begin{claim}\label{cla:topo}
The set $\Sigma=\{u\in \c^2\,|\; \|{\rm Re} (\vgot u)\|>\delta\; \forall \vgot \in \Zgot\}$ is path-connected. 

Furthermore, given $v,$ $w$ in $\Sigma\times\c$ there exists a smooth arc $c:[0,1]\to \Sigma\times\c$ such that $c(0)=v,$ $c(1)=w$ and $c'(t)\in\Theta$ $\forall t\in[0,1].$
\end{claim}
\begin{proof} Fix two different points $u_1,$ $u_2 \in \Sigma.$ Notice that  $\ell_u:=\{t u\,|\; t>1\}\subset \Sigma$ for all  $u \in \Sigma.$ 

Denote by $\s^3(R)$ the $3$-dimensional Euclidean sphere of radius $R>0$ in $\r^4\equiv \c^2$  and write $\s^3\equiv \s^3(1).$

For each $\vgot\in \Zgot,$ let $\gamma_\vgot\subset \s^3\equiv\s^3(1)$ denote the spherical geodesic $H_\vgot\cap \s^3,$ where $H_\vgot=\{u \in \c^2\,|\; {\rm Re} (\vgot u)=0\},$ and denote by $\Gamma=\cup_{\vgot\in\Zgot}\gamma_\vgot.$ Notice that $u_i/\|u_i\|\notin \Gamma$ $\forall i=1,2.$ Since $\frac1{R} (\s^3(R)-\Sigma)$ is the tubular neighborhood of $\Gamma$ in $\s^3$ given by $\cup_{\vgot\in\Zgot}\{u \in \s^3\,|\;  \|{\rm Re} (\vgot u)\|\leq\delta/R\},$ one has 
$$\lim_{R \to +\infty} \frac1{R} (\s^3(R)-\Sigma) =\Gamma$$ 
in the topology associated to the Hausdorff distance. Since $\s^3-\Gamma$ is path-connected and contains $u_1/\|u_1\|$ and $u_2/\|u_2\|,$ then these two points lie in the same connected component of $\frac1{R} (\s^3(R)\cap\Sigma)$ for a large enough $R.$ Equivalently,  $\frac{R}{\|u_1\|}u_1$ and $\frac{R}{\|u_2\|}u_2$ lie in the same path-connected component $\Omega$ of $\s^3(R)\cap\Sigma.$ Then, $\ell_{u_1}\cup\Omega\cup \ell_{u_2}$ is path-connected and so is $\Sigma.$

For the second part, since $\Sigma$ is open and path-connected, then there exists a polygonal arc $\hat{c}:[0,1]\to\Sigma\times\c$ connecting $v$ and $w$ and with $\hat{c}'(t)\in\Theta$ at any regular point. To finish, choose $c$ as a suitable smoothing of $\hat{c}.$  
\end{proof}

By Claim \ref{cla:topo} and equation \eqref{eq:lema}, one can construct a generalized null curve $G:S\to\c^3$ satisfying $G|_M=F$ and $\mini{G^\Zgot}_{\gamma}>\delta.$ From Lemma \ref{lem:runge} applied to $G$ and a continuity argument, we obtain a compact region  $U$ and a null curve $H\in{\sf N}(U)$ satisfying that
\begin{enumerate}[{\sf (i)}]
\item $S\subset U^\circ\subset U\subset V^\circ,$
\item $-\chi(V-U^\circ)=n-1,$
\item $\|H-G\|_{1,M}<\epsilon/2,$ and
\item $\mini{H^\Zgot}_{U-M^\circ}>\delta.$
\end{enumerate}

Since $-\chi(V-U^\circ)<n$, the induction hypothesis applied to $H$ and $\epsilon_0=\min\{\epsilon/2,\delta-\delta/\ngot\}$ gives a null curve $\hat{F}\in{\sf N}(V)$ such that
\begin{enumerate}[{\sf (I)}]
\item $\|\hat{F}-H\|_{1,U}<\epsilon_0,$
\item $\mini{\hat{F}^\Zgot}_{V-U^\circ}>\delta/\ngot,$ and
\item $\mini{\hat{F}^\Zgot}_{\partial V}>\kappa.$
\end{enumerate}

Then, {\sf (L1)} follows from {\sf (iii)} and {\sf (I)}. Properties {\sf (iv)}, {\sf (I)}, and {\sf (II)} give {\sf (L2)}. Finally, {\sf (III)} directly implies {\sf (L3)}. Hence, $\hat{F}$ satisfies the conclusion of Lemma \ref{lem:fun} and we are done.


\section{Main results}
Given an admissible compact region $M\subset \Ncal,$  $F\in{\sf N}(M),$ a finite subset $\Zgot\subset\s^1$ and $r>0,$ it is not hard to find $v\in \c^3$ so that the null curve $X=F+v$ satisfies that
$\mini{X^\Zgot}_{\partial M}>r.$ Indeed, it suffices to choose $v \in \c^3$ such that $\|F\|<\|v^\vgot\|-r$ on $M$ for all $\vgot \in \Zgot.$ Therefore, Main Theorem in the introduction follows from the following extension

\begin{theorem}\label{th:main}
Let $M$ be an admissible compact region in $\Ncal,$ let $\Zgot_0\subset\s^1$ be a finite subset with cardinal number $\ngot\in\n,$  and let $X\in{\sf N}(M)$ such that
\begin{equation}\label{eq:teo}
\mini{X^{\Zgot_0}}_{\partial M}>\ngot.
\end{equation}

Then, for any $\varepsilon>0,$ there exist an infinite closed subset $\Zgot_\Ncal\subset\s^1$ and $Y\in {\sf N}(\Ncal)$ such that
\begin{enumerate}[{\sf (A)}]
\item $\Zgot_0\subset \Zgot_\Ncal,$
\item $\|Y-X\|_{1,M}<\varepsilon,$
\item the map $\Ygot:\Zgot_\Ncal \times \Ncal\to\r^2$ given by $\Ygot(\vgot,P)=Y^\vgot(P)$ is proper, and
\item $\mini{Y^{\Zgot_0}}>1- \varepsilon$ on $\Ncal-M.$
\end{enumerate}
\end{theorem}
\begin{proof}
Without loss of generality, we assume that $\varepsilon<1.$ 

Let $\{M_n\,|\;n \in \{0\}\cup\n\}$ be an exhaustion  of $\Ncal$ by admissible compact regions with analytical boundary satisfying that $M_0=M$ and $M_{n-1} \subset M_{n}^\circ$ for all $n \in \n.$

Label $X_0=X$ and let us construct a sequence $\{(X_n,\Zgot_n)\}_{n \in \n}$ of null curves and finite subsets satisfying that
\begin{enumerate}[\sf (a$_n$)]
  \item $X_n \in {\sf N}(M_n)$ for all $n \in \n,$
  \item $\Zgot_{n-1}\subset\Zgot_n\subset \s^1$ and the cardinal number of $\Zgot_n$ is $\ngot+n,$ 
  \item $\|X_{n}-X_{n-1}\|_{1,M_{n-1}}< \varepsilon_n$ for all $n \in \n,$ where
  \[
  \varepsilon_n<\frac1{2^{n+1}}\min\big\{\varepsilon\,,\, \min\{\min_{M_k}\|\frac{dX_k}{\sigma_\Ncal}\|\,|\, k=0,\ldots,n-1\}\big\}>0
  \]
  (notice that $dX_k$ never vanishes on $M_k$ since $X_k\in{\sf N}(M_k)$),
  \item $\mini{X_k^{\Zgot_n}}_{M_k-M_{k-1}^\circ}>k$ for all $k\in\{1,\ldots,n\},$ $n\in \n,$ and
  \item $\mini{X_n^{\Zgot_n}}_{\partial M_n}>(n+1)(\ngot+n)$ for all $n\in\n.$
\end{enumerate}

The sequence is obtained in a recursive way. The couple $(X_0,\Zgot_0)$ trivially satisfies {\sf (a$_0$)} and {\sf (e$_0$)}, whereas {\sf (b$_0$)}, {\sf (c$_0$)}, and {\sf (d$_0$)} make no sense. Let $n\geq 1,$ assume we already have a couple $(X_{n-1},\Zgot_{n-1})$ satisfying the corresponding properties, and let us construct $(X_n,\Zgot_n).$ For $\varepsilon_n$ small enough, the null curve $X_n\in{\sf N}(M_n)$ given by Lemma \ref{lem:fun} applied to the data 
\[
(M,V,\Zgot,F,\delta,\epsilon,\kappa)=\big(M_{n-1},M_n,\Zgot_{n-1}, X_{n-1},n(\ngot+n-1),\varepsilon_n,(n+1)(\ngot+n)\big)
\]
satisfies {\sf (a$_n$)}, {\sf (c$_n$)}, 
\begin{equation}\label{eq:nn}
\text{$\mini{X_k^{\Zgot_{n-1}}}_{M_k-M_{k-1}^\circ}>k$ for all $k\in\{1,\ldots,n\},$ and $\mini{X_n^{\Zgot_{n-1}}}_{\partial M_n}>(n+1)(\ngot+n).$}
\end{equation}
For the first assertion in \eqref{eq:nn}, use items {\sf (c$_n$)} and {\sf (d$_{n-1}$)} and a continuity argument for $k\in\{1,\ldots,n-1\},$ and Lemma \ref{lem:fun}-{\sf (L2)} for $k=n.$

To close the induction, choose any $\vgot\in\s^1-\Zgot_{n-1}$ such that  the couple $(X_n,\Zgot_n:=\Zgot_{n-1}\cup\{\vgot\})$ satisfies {\sf (b$_n$)}, {\sf (d$_n$)}, and {\sf (e$_n$)}. Since  $M_k-M_{k-1}^\circ,$ $k=1,\ldots,n,$ and $\partial M_n$ are compact, the existence of such a $\vgot$ near $\Zgot_{0}$ is guaranteed by a continuity argument and \eqref{eq:nn}.
  
By items {\sf (a$_n$)} and {\sf (c$_n$)}, $\{X_n\}_{n\in \n}$ uniformly converges on compact subsets of $\Ncal$ to a
holomorphic map $Y:\Ncal \to \c^3$ with $\curvo{dY,dY}=0.$ Set $\Zgot_\infty:=\cup_{n\in\n}\Zgot_n,$ $\Zgot_\Ncal=\overline{\Zgot}_\infty,$ and let us show that the couple $(Y,\Zgot_\Ncal)$ satisfies the theses of the theorem.

From {\sf (b$_n$)}, $\Zgot_\Ncal$ is a closed infinite set. 
Item {\sf (A)} is obvious. 

To prove that $Y$ is an immersion, hence a null curve, it suffices to check that $\|dY/\sigma_\Ncal\|(P)>0$ $\forall P\in\Ncal.$ Indeed, let $P\in\Ncal$ and choose $j\in\n$ so that $P\in M_j.$ Then {\sf (c$_n$)} implies that 
\begin{eqnarray*}
\|dY/\sigma_\Ncal\|(P) & \geq & \|dX_{j}/\sigma_\Ncal\|(P) - \sum_{k\geq j} \|X_{k+1}-X_k\|_{1,M_k}\\
 & > & \|dX_{j}/\sigma_\Ncal\|(P) - \sum_{k\geq j} \varepsilon_{k+1} \\
 & > & \|dX_{j}/\sigma_\Ncal\|(P) - \sum_{k\geq j} \frac{1}{2^{k+1}} \|dX_{j}/\sigma_\Ncal\|(P) \\
 & \geq & \frac12 \|dX_{j}/\sigma_\Ncal\|(P)>0,
\end{eqnarray*}
hence $Y$ is an immersion as claimed.

From {\sf (c$_n$)} one has 
\begin{equation}\label{eq:cerca}
\text{$\|Y-X_j\|_{1,M_j} < \sum_{k=j+1}^\infty \|X_k-X_{k-1}\|_{1,M_{k-1}}<\sum_{k=j+1}^\infty \varepsilon_k<\varepsilon$ for all $j\geq 0,$}
\end{equation}
proving in particular {\sf (B)}.

To prove {\sf (C)}, take $k\in\n.$ From \eqref{eq:cerca} and {\sf (d$_j$)}, $j\geq k,$ one infers $\mini{Y^{\Zgot_j}}_{M_k-M_{k-1}^\circ}>k-\varepsilon$ for all $j\geq k.$ Therefore, $\Ygot\big(\Zgot_\infty\times (M_k-M_{k-1}^\circ)\big)\cap B(k-\varepsilon)=\emptyset$ for all $k\in\n,$ hence $\Ygot^{-1}(\overline{B(k-2\varepsilon)})\cap(\Zgot_\infty\times\Ncal)\subset \Zgot_\infty\times M_{k-1}$ is relatively compact in $\Zgot_\Ncal\times\Ncal.$ Thus $\Ygot^{-1}(\overline{B(k-2\varepsilon)})$ is compact in $\Zgot_\Ncal\times\Ncal$ for all $k,$ proving {\sf (C)}.

Finally, {\sf (D)} follows from {\sf (b$_n$)}, {\sf (d$_n$)}, and \eqref{eq:cerca}.
\end{proof}

\begin{remark} 
Theorem \ref{th:main} shows that $\Zgot_0$ is a universal projector set and $\Zgot_\Ncal\subset \s^1$ is a projector set for the fixed $\Ncal.$ Since obviously $\Zgot_\Ncal$  depends on $\Ncal,$ one can not infer that it is a universal projector set. 

On the other hand, up to an elementary refinement of the proof of Theorem \ref{th:main}, one can construct $\Zgot_\infty$ to be closed, and even with accumulation set in $\Zgot_0.$ In this case, $\Zgot_\Ncal=\Zgot_\infty$ is countably infinite.
\end{remark}

The following proposition shows that $\s^1$ is a projector set for no hyperbolic Riemann surface. 

\begin{proposition} \label{pro:sharp}
Let $M$ be an open Riemann surface and let $F=(F_1,F_2):M \to \c^2$ be a holomorphic map. Assume that the map $\Psi:\s^1\times M \to \r^2,$ $\Psi (\vgot,P)={\rm Re}(\vgot F),$ is proper.

Then $M$ is parabolic. As a consequence, if $M$ has finite topology then $M$ is biholomorphic to a finitely punctured compact Riemann surface and $F_1$ and $F_2$ extend meromorphically to the compactification of $M.$
\end{proposition}
\begin{proof} To show that $M$ is parabolic, it suffices to check that $F_1:M\to\c$ (and likewise $F_2$) is a proper holomorphic function. Reason by contradiction and take a divergent sequence $\{P_n\}_{n \in \n} \subset M$ such that $\{F_1(P_n)\}_{n \in \n}$ is bounded. For each $n\in \n$ choose $\vgot_n \in \s^1$ such that ${\rm Re}(\vgot_n F_2(P_n))=0.$  Then $\{\Psi(\vgot_n,P_n)\}_{n \in \n}$ is bounded as well, which is absurd.

For the second part of the proposition, assume that $M$ has finite topology. The parabolicity implies that $M=\overline{M}-\{Q_1,\ldots,Q_k\},$ where $\overline{M}$ is a compact Riemann surface  and $Q_1,\ldots,Q_k \in \overline{M}.$  Since $F_1,$ $F_2:M \to \c$ are proper holomorphic functions then they have no essential singularities at the ends, and so, they extend meromorphically to $\overline{M}.$
\end{proof}

\begin{corollary} \label{co:sharp}
Let $M$ be an open Riemann surface of finite topology and let $F:M \to \c^3$ be a null curve. Assume that the map $\Ygot:\s^1\times M \to \r^2,$ $\Ygot (\vgot,P)=F^\vgot(P),$ is proper.

Then $F$ has finite total curvature.  
\end{corollary}
\begin{proof}
Just write $F=(F_j)_{j=1,2,3},$ take into account that $dF_1$ and $dF_2,$ hence $dF_3,$ extend meromorphically to the natural compactification of $M,$ and Osserman's classical results \cite{Os}.
\end{proof}


\end{document}